\documentclass[12pt,reqno]{amsart}

\newcommand\version{February 1, 2022}

\usepackage{amsmath,amsfonts,amsthm,amssymb,amsxtra}
\usepackage{bbm} 

\newtheorem{theorem}{Theorem}
\newtheorem{proposition}[theorem]{Proposition}
\newtheorem{lemma}[theorem]{Lemma}
\newtheorem{corollary}[theorem]{Corollary}

\theoremstyle{definition}

\theoremstyle{remark}

\newcommand{\1}{\mathbbm{1}}

\newcommand{\cl}{\mathrm{cl}}

\renewcommand{\epsilon}{\varepsilon}

\newcommand{\loc}{{\rm loc}}

\renewcommand{\phi}{\varphi}
\newcommand{\R}{\mathbb{R}}

\DeclareMathOperator{\dist}{dist}

\DeclareMathOperator{\Tr}{Tr}
\DeclareMathOperator{\tr}{Tr}

\begin{document}

\title[Weyl's law under minimal assumptions --- \version]{Weyl's law under minimal assumptions}

\author{Rupert L. Frank}
\address[Rupert L. Frank]{Mathe\-matisches Institut, Ludwig-Maximilans Universit\"at M\"unchen, The\-resienstr.~39, 80333 M\"unchen, Germany, and Munich Center for Quantum Science and Technology, Schel\-ling\-str.~4, 80799 M\"unchen, Germany, and Mathematics 253-37, Caltech, Pasa\-de\-na, CA 91125, USA}
\email{r.frank@lmu.de}

\renewcommand{\thefootnote}{${}$} \footnotetext{\copyright\, 2022 by the author. This paper may be reproduced, in its entirety, for non-commercial purposes.\\
Partial support through U.S.~National Science Foundation grant DMS-1954995 and through the German Research Foundation grant EXC-2111-390814868 is acknowledged.}

\dedicatory{Dedicated to the memory of Sergey Naboko}

\begin{abstract}
	We show that Weyl's law for the number and the Riesz means of negative eigenvalues of Schr\"odinger operators remains valid under minimal assumptions on the potential, the vector potential and the underlying domain.
\end{abstract}

\maketitle

\section{Introduction and main result}

The celebrated Weyl law states that
\begin{equation}
	\label{eq:weylintro}
	\lim_{h\to 0_+} h^d\, N(0,-h^2\Delta_\Omega +V) = L_{0,d}^\cl \int_\Omega V(x)_-^{d/2}\,dx
\end{equation}
where $a_\pm :=\max\{\pm a,0\}$, where $L_{0,d}^\cl$ is a special case of
$$
L_{\gamma,d}^\cl := (4\pi)^{-d/2} \frac{\Gamma(\gamma+1)}{\Gamma(\gamma+d/2+1)}
$$
and where $N(0,H)$ denotes the number of negative eigenvalues, counting multiplicities, of a selfadjoint operator $H$. We restrict ourselves here to the case where the Laplacian $-\Delta_\Omega$ is considered with Dirichlet boundary conditions.

Weyl proved \eqref{eq:weylintro} in the case where $V$ is a negative constant and $\Omega$ a bounded, open set with sufficiently smooth boundary. Following his work, the asymptotics were extended to a larger class of $\Omega$'s and $V$'s. In the late 60's and early 70's Birman and Solomyak emphasized the problem of proving \eqref{eq:weylintro} under minimal assumptions on $\Omega$ and $V$. In fact, they also considered more general problems involving higher order elliptic differential operators with variable coefficients, but this is beyond the scope of our discussion here. Based, in part, on their work, Rozenblum \cite{Ro0,Ro} succeeded in proving \eqref{eq:weylintro} for $V$ equal to a negative constant under the sole assumption that $\Omega\subset\R^d$ is an open set of finite measure. In \cite{Ro2,Ro3} he extended this result to $V\in L^{d/2}(\Omega)$ for $d\geq 3$, which is the optimal assumption on $V_-$ for \eqref{eq:weylintro} to be finite.

Our modest contribution in this paper is to show that the assumption on $V_+$ can be relaxed to $V_+\in L^1_\loc(\Omega)$. Thus, we show that even wild positive singularities cannot weaken the influence of $V_-$ to \eqref{eq:weylintro}. We emphasize that the local integrability assumption means that $V_+$ is integrable on any compact set contained in $\Omega$ and, in particular, it does not restrict in any way the behavior of $V$ near infinity or, if $\Omega\neq\R^d$, near the boundary of $\Omega$.

In fact, we prove the corresponding result also for the Riesz means $\Tr \left( -h^2\Delta_\Omega + V \right)_-^\gamma$ and, to simplify various statement, we interpret this quantity in the case $\gamma=0$ as $N(0,-h^2\Delta_\Omega+V)$.

Technically, under the assumption $V\in L^1_\loc(\Omega)$ with $V_-\in L^{\gamma+d/2}(\Omega)$, where $\gamma\geq 1/2$ if $d=1$, $\gamma>0$ if $d=2$ and $\gamma\geq 0$ if $d\geq 3$, the operator $-h^2\Delta_\Omega +V$ is defined as a selfadjoint operator in $L^2(\Omega)$ via the quadratic form
$$
\int_\Omega \left( h^2 |\nabla \psi|^2 + V|\psi|^2 \right)dx
$$
with form domain $\{ \psi\in H^1_0(\Omega):\  V_+ |u|^2 \in L^1(\Omega) \}$. By Sobolev embedding theorems, the above conditions on $\gamma$ guarantee that this quadratic form is lower semibounded and closed; see, e.g., \cite[Section 4.1]{FrLaWe}.

Our main result is the following.

\begin{theorem}\label{main}
	Let $\gamma\geq 1/2$ if $d=1$, $\gamma>0$ if $d=2$ and $\gamma\geq 0$ if $d\geq 3$. Let $\Omega\subset\R^d$ be an open set and let $V\in L^1_\loc(\Omega)$ with $V_-\in L^{\gamma+d/2}(\Omega)$. Then
	\begin{equation}
		\label{eq:main}
		\lim_{h\to 0_+} h^d \Tr \left( -h^2\Delta_\Omega + V \right)_-^\gamma = L_{\gamma,d}^\cl \int_\Omega V(x)_-^{\gamma+d/2}\,dx \,,
	\end{equation}
	where $-\Delta_\Omega + V$ is considered in $L^2(\Omega)$ with Dirichlet boundary conditions.
\end{theorem}

\emph{Remarks.} (a) For $\gamma=0$, the assumption $d\geq 3$ is necessary. Indeed, as shown in \cite{BiLa,NaSo} for $d=1,2$ there are $V\in L^{d/2}(\Omega)$ such that $-h^2\Delta_\Omega+V$ can be defined and such that \eqref{eq:main} fails. It is unknown whether such $V$ exist also for $0<\gamma<1/2$ in $d=1$. For constant $V$ and $\Omega$ of finite measure, the theorem holds for all $\gamma\geq 0$, as is well-known \cite[Corollaries 3.15 and 3.17]{FrLaWe}.\\
(b) Our proof shows that the inequality $\geq$ in \eqref{eq:main} with $\lim$ replaced by $\liminf$ holds for \emph{all} $\gamma\geq 0$, provided the operators $-h^2\Delta_\Omega+V$ are well-defined.\\
(c) We reiterate that under the additional assumption $V_+\in L^{\gamma+d/2}(\Omega)$ this theorem can be considered known. For $\gamma=0$ it is due to Rozenblum \cite{Ro2,Ro3}, and for $\gamma>0$ one can easily modify his arguments. Our contribution is to weaken the assumption on $V_+$ to $V_+\in L^1_\loc(\Omega)$.\\
(d) We do not see how to obtain our result using the method of Birman--Solomyak and Rozenblum. Instead we use, at least for $\gamma=1$, the method of coherent states. The results for $\gamma\neq 1$ are deduced from that for $\gamma=1$ by relatively soft arguments.\\
(e) In connection with Weyl's law the method of coherent states is used, for instance, in Berezin's work \cite{Be}. In the context of Schr\"odinger operators it appears, for instance, in \cite{Th,Li}. Closest to our work is probably \cite{EvLeSiSo}, who use a weak convergence argument to deduce the $\gamma=0$ case from the $\gamma=1$ case. This weak convergence argument is, in part, inspired by \cite{LiSi}. We show more generally that the $\gamma-1$ case can be deduce from the $\gamma$ case, provided $\gamma\geq 1$. This generalization is not really necessary for the proof of Theorem \ref{main}, but it is for its extension mentioned in (h) below.\\
(f) Our theorem concerns leading order semiclassics without a remainder. As an aside, we note that the method of coherent states can be modified to derive very good remainder estimates under additional assumptions on $V$ \cite{SoSp}. For its use to derive subleading terms for constant $V$ in sufficiently regular $\Omega$, see, for instance, \cite{FrGe1,FrGe2,FrLa}.\\
(g) The proof of Theorem \ref{main} splits naturally into two parts, namely,
\begin{equation}
	\label{eq:mainlower}
	\liminf_{h\to 0_+} h^d \Tr \left( -h^2\Delta_\Omega + V \right)_-^\gamma \geq L_{\gamma,d}^\cl \int_\Omega V(x)_-^{\gamma+d/2}\,dx \,,
\end{equation}
and
\begin{equation}
	\label{eq:mainupper}
	\limsup_{h\to 0_+} h^d \Tr \left( -h^2\Delta_\Omega + V \right)_-^\gamma \leq L_{\gamma,d}^\cl \int_\Omega V(x)_-^{\gamma+d/2}\,dx \,.
\end{equation}
The difficulty coming from the weak assumption $V_+\in L^1_\loc(\Omega)$ on the potential arises only in the proof of \eqref{eq:mainlower}, since for \eqref{eq:mainupper} we can bound $\Tr(-h^2\Delta_\Omega +V)_-^\gamma\leq \Tr(-h^2\Delta_\Omega -V_-)_-^\gamma$ and use known results for the asymptotics of the latter, see (c). However, we have chosen to also present a proof of \eqref{eq:mainupper} to advertize the technique in (e).\\
(h) Theorem \ref{main} remains valid in the presence of a magnetic field. This will be discussed in Section \ref{sec:magnetic}.

\begin{corollary}\label{maincor}
	Let $\gamma\geq 0$, let $\Omega\subset\R^d$ be an open set and let $V\in L^1_\loc(\Omega)$. Assume that $V_-$ is infinitesimally form bounded with respect to $-\Delta_\Omega$ and that
	$$
	\liminf_{h\to 0_+} h^d \Tr \left( -h^2\Delta_\Omega + V \right)_-^\gamma <\infty \,.
	$$
	Then $V_-\in L^{\gamma+d/2}(\Omega)$. In particular, if $\gamma$ is as in Theorem \ref{main}, then \eqref{eq:main} holds.
\end{corollary}

\begin{proof}
	Let $R,M>0$ and set
	\begin{equation}
		\label{eq:defvm}
		V_M(x):=\max\{ V(x),-M\} \,.
	\end{equation}
	Then, by the variational principle,
	$$
	\Tr \left( -h^2\Delta_\Omega + V \right)_-^\gamma \geq \Tr \left( -h^2\Delta_\Omega + \1_{\Omega\cap B_R} V_M \right)_-^\gamma \,.
	$$
	We multiply this inequality by $h^d$ and take the $\liminf_{h\to 0_+}$. On the left side, by assumption, we get a constant $C<\infty$, say. On the right side, we apply Remark (b) following Theorem \ref{main}, noting that $\1_{\Omega\cap B_R} V_M$ satisfies its assumptions. Thus, we obtain
	$$
	C \geq L_{\gamma,d}^\cl \int_{\Omega\cap B_R} \left( V_M(x) \right)_-^{\gamma+d/2}\,dx \,.
	$$
	Since $R,M>0$ are arbitrary, monotone convergence gives $V_-\in L^{\gamma+d/2}(\Omega)$. In particular, if $\gamma$ is as in Theorem \ref{main}, then we can apply this theorem and obtain \eqref{eq:main}.
\end{proof}

The remainder of this paper is organized as follows. In Sections \ref{sec:1}, \ref{sec:>1} and \ref{sec:<1} we prove Theorem \ref{main} in the cases $\gamma=1$, $\gamma>1$ and $\gamma<1$, respectively. In particular, the weak convergence statement alluded to above appears in Subsection \ref{sec:weakconv}. Section \ref{sec:magnetic} is devoted to an extension to the magnetic case.

It is my pleasure to thank Timo Weidl for his interest and encouragement to write this note. I am also very grateful to Grigori Rozenblum for helpful remarks.

I would like to dedicate this paper, with great respect, to the memory of Sergey Naboko. I was fortunate to have had the opportunity to interact with him, resulting in a joint paper \cite{FrHaNaSe}, and I am very grateful for all the ideas, support and advice he has shared with me, starting from the time when I took my first step in research.


\section{The case $\gamma=1$}\label{sec:1}

Our goal in this section is to prove Theorem \ref{main} in the special case $\gamma=1$. The proof splits naturally in a lower and an upper bound on $\Tr(-h^2\Delta_\Omega +V)_-$, which will be the topic of Subsections \ref{sec:upper} and \ref{sec:lower}, respectively. Both bounds rely on the method of coherent states, which we briefly review in Subsection \ref{sec:cs}.


\subsection{Basic properties of coherent states}\label{sec:cs}

Let $g\in L^2(\R^d)$ be a real function with $\|g\|=1$. The family of functions $(G_{p,q})_{p,q\in\R^d}$, defined by
\begin{equation}
	\label{eq:cohst}
	G_{p,q}(x) := e^{ip\cdot x} g(x-q) \,,
	\qquad x\in\R^d \,,
\end{equation}
is called \emph{coherent states} (associated to $g$). For a function $V$ on $\R^d$ we abbreviate
\begin{equation}
	\label{eq:meanv}
	\langle V \rangle (q):= \int_{\R^d} V(x) g(x-q)^2\,dx \,,
	\qquad q\in\R^d \,,
\end{equation}
whenever this integral converges.

We use the following simple properties of coherent states; see also \cite[Chapter 12]{LiLo}.

\begin{lemma}
	\label{cslowersymb}
	Let $p,q\in\R^d$. Then $G_{p,q}\in L^2(\R^d)$ with
	\begin{equation}
		\label{eq:cslowernorm}
		\|G_{p,q}\| = 1 \,.
	\end{equation}
	If $g\in L^\infty(\R^d)$ with compact support and $V\in L^1_\loc(\R^d)$, then
	\begin{equation}
		\label{eq:cslowerpot}
		\int_{\R^d} V(x) |G_{p,q}(x)|^2 \,dx = \langle V \rangle (q) \,.
	\end{equation}
	If $g\in H^1(\R^d)$, then
	\begin{equation}
		\label{eq:cslowerkin}
		\| \nabla G_{p,q}\|^2 = p^2 + \|\nabla g\|^2 \,.
	\end{equation}
\end{lemma}

\begin{lemma}
	\label{csuppersymb}
	Let $f\in L^2(\R^d)$. Then
	\begin{equation}
		\label{eq:csidentity}
		\iint_{\R^d\times\R^d} \left| (G_{p,q},f)\right|^2 \,\frac{dp\,dq}{(2\pi)^d} = \|f\|^2 \,.
	\end{equation}
	Moreover,
	\begin{equation}
		\label{eq:csupperpot}
		\iint_{\R^d\times\R^d} V(q) \left| (G_{p,q},f)\right|^2 \,\frac{dp\,dq}{(2\pi)^d} = \int_{\R^d} (V*|g|^2)(x) |f(x)|^2 \,dx \,.
	\end{equation}
	If, in addition, $g\in H^1(\R^d)$ is real-valued and $f\in H^1(\R^d)$, then
	\begin{equation}
		\label{eq:csupperkin}
		\iint_{\R^d\times\R^d} p^2 \left| (G_{p,q},f)\right|^2 \,\frac{dp\,dq}{(2\pi)^d} = \|\nabla f\|^2 + \|\nabla g\|^2 \|f\|^2 \,.
	\end{equation}
\end{lemma}


\subsection{Lower bound on $\Tr(-h^2\Delta_\Omega+V)_-$}\label{sec:upper}

Let $\omega\subset\R^d$ be a bounded, open set with $\overline\omega\subset\Omega$. Let $g\in H^1(\R^d)\cap L^\infty(\R^d)$ be a real function with $\|g\|_2=1$ and support contained in the open ball centered at the origin of radius $\dist(\omega,\R^d\setminus\Omega)>0$. We consider the coherent states $G_{p,q}$ corresponding to this function $g$ as defined in Subsection \ref{sec:cs}. With a parameter $M>0$ and recalling the definition of $V_M$ in \eqref{eq:defvm}, we choose the following trial matrix in the variational principle for the sum of eigenvalues,
$$
\Gamma := \iint_{\R^d\times\omega} \1_{\{h^2p^2 + V_M(q)<0\}} |G_{p,q}\rangle\langle G_{p,q}| \,\frac{dp\,dq}{(2\pi)^d} \,.
$$
Clearly, $0\leq\Gamma\leq 1$ and the range of $\Gamma$ lies in $H^1_0(\Omega)\cap L^2(\Omega,V_+\,dx)$. Thus, by the variational principle and using the computations from Lemma \ref{cslowersymb} (note that these computations are applicable since $V_+$ is multiplied by functions supported on a compact subset of $\Omega$),
\begin{align*}
	& -\tr\left( -h^2\Delta_\Omega +V\right)_- \leq \tr \left( -h^2\Delta_\Omega + V \right)\Gamma \\
	& \quad = \iint_{\R^d\times\omega} \1_{\{h^2p^2 + V_M(q)<0\}} \left( \int_{\R^d} \left( h^2 |\nabla G_{p,q}|^2 + V |G_{p,q}|^2 \right) dx \right) \frac{dp\,dq}{(2\pi)^d} \\
	& \quad = \iint_{\R^d\times\omega} \1_{\{h^2p^2 + V_M(q)<0\}} \left( h^2 p^2 + h^2 \|\nabla g\|_2^2 + \langle V \rangle(q) \right) \frac{dp\,dq}{(2\pi)^d} \\
	& \quad = \iint_{\R^d\times\omega} \1_{\{h^2p^2 + V_M(q)<0\}} \left( h^2p^2 + V_M(q) \right) \frac{dp\,dq}{(2\pi)^d} + h^{-d} \mathcal R \\
	& \quad = - h^{-d} L_{1,d}^\cl \int_{\omega} V_M(q)_-^{1+d/2} \,dq +h^{-d} \mathcal R \,.
\end{align*}
Here we have set
\begin{align*}
	\mathcal R & := h^d \iint_{\R^d\times\omega} \1_{\{h^2p^2 + V_M(q)<0\}} \left( h^2 \|\nabla g\|_2^2 + \langle V \rangle(q) - V_M(q) \right) \frac{dp\,dq}{(2\pi)^d} \\
	& = L_{0,d}^\cl \int_\omega V_M(q)_-^{d/2} \left( h^2 \|\nabla g\|_2^2 + \langle V \rangle(q) - V_M(q) \right) dq \,.
\end{align*}
It is convenient to decompose this remainder term further,
$$
\mathcal R = \mathcal R_1 + \mathcal R_2 \,,
$$
where
\begin{align*}
	\mathcal R_1 & := h^{2} L_{0,d}^\cl \|\nabla g\|_2^2 \int_{\omega} V_M(q)_-^{d/2} \,dq \,,\\
	\mathcal R_2 & := L_{0,d}^\cl \int_\omega V_M(q)_-^{d/2} \left( \langle V \rangle(q) - V_M(q) \right) dq \,.
\end{align*}
The integral in $\mathcal R_1$ is finite since $(V_M)_-\leq V_-\in L_\loc^{d/2}(\Omega)$. Thus, we obtain
$$
\liminf_{h\to 0_+} h^d \tr\left( -h^2\Delta_\Omega +V\right)_- \geq L_{1,d}^\cl \int_{\omega} V_M(q)_-^{1+d/2} \,dq - \mathcal R_2 \,.
$$

We show now that by a suitable choice of $g$ (with $\omega$ and $M$ fixed), we can make $\mathcal R_2$ smaller than any given positive constant. Once we have shown this, we obtain
$$
\liminf_{h\to 0_+} h^d \tr\left( -h^2\Delta_\Omega +V\right)_- \geq L_{1,d}^\cl \int_{\omega} V_M(q)_-^{1+d/2} \,dq \,,
$$
and then we can let $M\to+\infty$ and take a sequence of $\omega$'s that increase to $\Omega$ to obtain the claimed bound.

To prove the statement about $\mathcal R_2$, we decompose it further,
$$
\mathcal R_2 = \mathcal R_2' + \mathcal R_2'' + \mathcal R_2''' \,,
$$
where
\begin{align*}
	\mathcal R_2' & := - L_{0,d}^\cl \int_\omega V_M(q)_-^{d/2} \left( \langle V_- \rangle(q) - V(q)_- \right) dq \,, \\
	\mathcal R_2'' & := - L_{0,d}^\cl \int_\omega V_M(q)_-^{d/2} \left( V_M(q) + V(q)_- \right)dq \,,\\
	\mathcal R_2''' & := L_{0,d}^\cl \int_\omega V_M(q)_-^{d/2} \langle V_+ \rangle(q)\,dq \,.
\end{align*}
Since $V_M + V_- \geq 0$, we have $\mathcal R_2''\leq 0$ and we can drop it from the further discussion. We bound
$$
\left| \mathcal R_2' \right| \leq L_{0,d}^\cl \left\| (V_M)_- \right\|_{1+d/2}^{d/2} \left\| \langle V_- \rangle - V_- \right\|_{1+d/2} \leq L_{0,d}^\cl \left\| V_- \right\|_{1+d/2}^{d/2} \left\| \langle V_- \rangle - V_- \right\|_{1+d/2} \,.
$$
We now choose $g$ of the form $g_\delta(x) := \delta^{-d/2} G(x/\delta)$, where $G\in H^1(\R^d)\cap L^\infty(\R^d)$ is a real function with $\|G\|=1$ and support in the closed unit ball. We will assume that $0<\delta\leq (1/2)\dist(\omega,\R^d\setminus\Omega)$, so that $g_\delta$ satisfies all the assumptions on $g$ above. For emphasis, we will use the notation $\langle \cdot \rangle_\delta$. Since $V_-\in L^{1+d/2}(\R^d)$ (we extend $V$ by zero to $\R^d\setminus\Omega$), it follows from standard properties of convolutions that $\langle V_- \rangle_\delta \to V_-$ in $L^{1+d/2}(\R^d)$. Thus, $\lim_{\delta\to 0} \mathcal R_2'=0$.

Note that the function $\langle V_+ \rangle_\delta$ on $\omega$ only depends on $V_+$ restricted to a neighborhood of $\omega$ of size $(1/2)\dist(\omega,\R^d\setminus\Omega)$. Since $V_+$ restricted to this neighborhood is integrable, it follows from standard properties of convolutions that $\langle V_+ \rangle_\delta\to V_+$ in $L^1(\omega)$. Since $(V_M)_-^{d/2} \leq M^{d/2} \in L^\infty(\omega)$, it follows that $\lim_{\delta\to 0} \mathcal R_2'''= L_{0,d}^\cl \int_\omega V_M(q)^{d/2} V_+(q)\,dq = 0$.

Thus, we have shown that $\limsup_{\delta\to 0_+} \mathcal R_2 \leq 0$, which, as we explained, completes the proof.


\subsection{Upper bound on $\Tr(-h^2\Delta_\Omega+V)_-$}\label{sec:lower}

By the variational characterization of sums of negative eigenvalues, we need to prove a lower bound on $\tr(-h^2\Delta_\Omega +V)\Gamma$ for any operator $0\leq\Gamma\leq 1$ on $L^2(\Omega)$ with finite energy. Since this quantity does not increase if we replace $V$ by $-V_-$, we may assume from the outset that $V_+\equiv 0$.

We will denote by $\rho_\Gamma$ the density of a bounded, nonnegative operator $\Gamma$ on $L^2(\Omega)$ with finite kinetic energy $\tr(-\Delta_\Omega)\Gamma$. In the special case where $\Gamma$ is trace class (which is all that is needed for us) with Schmidt decomposition $\Gamma=\sum_n \nu_n |\psi_n\rangle\langle \psi_n|$, one has $\rho_\Gamma=\sum_n \nu_n |\psi_n|^2$. One easily verifies that, in the presence of degenerate eigenvalues, this is independent of the choice of the basis of eigenfunctions.

In order to better explain the strategy of the proof, we first assume \emph{that $\Omega$ has finite measure}. Because of the upper bound that we have already proved we may assume, in addition, that $\tr(-h^2\Delta+V)\Gamma\leq 0$. Let $g\in H^1(\R^d)$ be real and consider again the coherent states $(G_{p,q})$ from Section \ref{sec:cs}. According to Lemma \ref{csuppersymb}, we have
\begin{align*}
	\tr(-h^2\Delta+V)\Gamma = \iint_{\R^d\times\R^d} \left( h^2 p^2 + V(q) \right) (G_{p,q},\Gamma\, G_{p,q}) \, \frac{dp\,dq}{(2\pi)^d} + \mathcal R_1 + \mathcal R_2
\end{align*}
with
\begin{align*}
	\mathcal R_1 & := -h^2 \|\nabla g\|^2 \tr\Gamma = -h^2 \|\nabla g\|^2 \int_{\Omega} \rho_\Gamma \,dx \,, \\
	\mathcal R_2 & := \int_{\R^d} \left( V - V*|g|^2 \right) \rho_\Gamma \,dx \,.
\end{align*}
Since $0\leq\Gamma\leq 1$, we have $0\leq (G_{p,q},\Gamma \, G_{p,q}) \leq \|G_{p,q}\|^2 = 1$ by \eqref{eq:cslowernorm}. This implies
\begin{align*}
	\iint_{\R^d\times\R^d} \left( h^2 p^2 + V(q) \right) (G_{p,q},\Gamma \, G_{p,q}) \, \frac{dp\,dq}{(2\pi)^d} & \geq - \iint_{\R^d\times\R^d} \left( h^2 p^2 + V(q) \right)_- \, \frac{dp\,dq}{(2\pi)^d} \\
	& = - h^{-d} L_{1,d}^\cl \int_{\Omega} V_-^{1+d/2} \,dq \,.
\end{align*}
Thus, to complete the proof we have to show that $\mathcal R_1$ and $\mathcal R_2$ are small compared with this term. It is at this point that the Lieb--Thirring inequality enters in a crucial way. Indeed, combining this inequality with our a-priori assumption $\tr(-h^2\Delta_\Omega+V)\Gamma\leq 0$ and H\"older's inequality, we obtain
$$
h^2 K_d \int_{\R^d} \rho_\Gamma^{1+2/d} \,dx
\leq \tr(-h^2\Delta) \Gamma \leq - \int_{\R^d} V \rho_\Gamma \,dx \leq \|V_-\|_{1+d/2} \|\rho_\Gamma\|_{1+2/d} \,,
$$
that is,
$$
\int_{\R^d} \rho_\Gamma^{1+2/d} \,dx \leq h^{-d-2} K_d^{-(d+2)/2} \int_{\R^d} V_-^{1+d/2} \,dx \,.
$$
For the Lieb--Thirring inequality, see, e.g., \cite[Remark 7.21]{FrLaWe}. Note that the extension of an operator $\Gamma$ on $L^2(\Omega)$ with $\tr(-\Delta_\Omega)\Gamma<\infty$ by zero to an operator on $L^2(\R^d)$ satisfies $\tr(-\Delta)\Gamma=\tr(-\Delta_\Omega)\Gamma$.

We now use this a-priori bound to estimate our remainder terms. We obtain
$$
\mathcal R_1 \geq - h^2  \|\nabla g\|^2 |\Omega|^{2/(d+2)} \|\rho_\Gamma\|_{1+2/d} \geq - h^{-d+2} K_d ^{-d/2} \|\nabla g\|^2 |\Omega|^{2/(d+2)} \|V_-\|_{1+d/2}^{d/2}
$$
(here we use $|\Omega|<\infty$) and
$$
|\mathcal R_2| \leq \|V*|g|^2-V\|_{1+d/2} \|\rho_\Gamma\|_{1+2/d} \leq h^{-d} K_d ^{-d/2}  \|V*|g|^2-V\|_{1+d/2} \|V_-\|_{1+d/2}^{d/2} \,.
$$
Note that these bounds are independent of $\Gamma$. Thus, we obtain
\begin{align*}
	\limsup_{h\to 0_+} h^{-d} \tr\left( -h^2\Delta_\Omega +V\right)_- 
	& \leq  L_{1,d}^\cl \int_{\Omega} V(q)_-^{1+d/2} \,dq \\
	& \quad + K_d ^{-d/2}  \|V*|g|^2-V\|_{1+d/2} \|V_-\|_{1+d/2}^{d/2} \,.
\end{align*}
This bound is of the same form as in the upper bound in the first part of the proof. Therefore, by the same argument (involving a $g$ depending on a small parameter $\delta$), we finally obtain
\begin{align*}
	\limsup_{h\to 0_+} h^{-d} \tr\left( -h^2\Delta_\Omega +V\right)_- 
	\leq L_{1,d}^\cl \int_{\Omega} V(q)_-^{1+d/2} \,dq \,,
\end{align*}
which completes the proof of the lemma in the case $|\Omega|<\infty$.

Finally, we extend the result to arbitrary open $\Omega$. Given $R_->0$, we choose two smooth, real-valued functions $\chi_<$ and $\chi_>$ on $\R^d$ such that $\chi_<^2+\chi_>^2\equiv 1$, such that $\chi_<$ has compact support and such that $\chi_<\equiv 1$ on $B_{R_-}$. Consider again $0\leq\Gamma\leq 1$ with $\tr(-h^2\Delta_\Omega+V)\Gamma\leq 0$. Then, by the IMS localization formula,
\begin{align*}
	\tr(-h^2\Delta_\Omega+V)\Gamma = & \tr\left(-h^2\Delta_\Omega+V_h\right)\chi_<\Gamma\chi_< + \tr\left(-h^2\Delta_\Omega+V_h \right)\chi_>\Gamma\chi_>
\end{align*}
with
\begin{equation}
	\label{eq:imspot}
	V_h := V-h^2\left(|\nabla\chi_<|^2+|\nabla\chi_>|^2\right).
\end{equation}

We shall prove lower bounds on both terms on the right side separately.
Let $R_+>0$ be such that the support of $\chi_<$ is contained in the ball $B_{R_+}$ and put $\omega :=\Omega\cap B_{R_+}$. Then, by the variational principle, since $0\leq\chi_<\Gamma\chi_<\leq 1$,
\begin{align*}
	\tr\left(-h^2\Delta_\Omega+V_h \right)\chi_<\Gamma\chi_< \geq - \tr\left(-h^2\Delta_\omega+\1_\omega V_h \right)_- \,.
\end{align*}
We claim that
$$
\limsup_{h\to 0_+} h^d \tr\left(-h^2\Delta_\omega+\1_\omega V_h \right)_- \leq L_{1,d}^\cl \int_\Omega V_-^{1+d/2} \,dx \,.
$$
Indeed, this follows from the lower bound for sets of finite measure, even with the integral on the right side restricted to $\omega$. (To get rid of the $h$-dependence, fix $h_0>0$ and bound $V_h \geq V_{h_0}$ for $h\leq h_0$. Then let first $h\to 0_+$ using the result for sets of finite measure and then let $h_0\to 0_+$.)

Therefore, it remains to treat the `outside' contribution. According to the Lieb--Thirring inequality (see, e.g., \cite[Theorem 4.38]{FrLaWe}),
\begin{align*}
	\tr\left(-h^2\Delta_\Omega+V_h \right)\chi_>\Gamma\chi_> & \geq - \tr(-h^2\Delta_\Omega + \1_{\Omega\cap B_{R_-}^c} V_h)_- \\
	& \geq - h^{-d} L_{1,d} \int_{\Omega\cap B_{R_-}^c} \left(V_h \right)_-^{1+d/2} \,dx \,.
\end{align*}
Note that this is a lower bound independent of $\Gamma$. Finally, we notice that
$$
\lim_{h\to 0_+} \int_{\Omega\cap B_{R_-}^c} \left(V_h \right)_-^{1+d/2} \,dx = \int_{\Omega\cap B_{R_-}^c} V_-^{1+d/2} \,dx \,,
$$
which, by monotone convergence, can be made arbitrarily small by choosing $R_-$ large. This completes the proof.


\section{The case $\gamma>1$}\label{sec:>1}

Theorem \ref{main} for $\gamma>1$ follows in a relatively straightforward way from the case $\gamma=1$ of Theorem \ref{main} and the Lieb--Thirring inequality. Here are the details.

The starting point is the formula, valid for $\gamma>1$,
\begin{equation}
	\label{eq:representation}
	\Tr \left( -h^2\Delta_\Omega + V \right)_-^\gamma
	= \gamma(\gamma-1) \int_0^\infty \Tr \left( -h^2\Delta_\Omega + V + \kappa \right)_- \kappa^{\gamma-2}\,d\kappa \,.
\end{equation}
For $\kappa>0$ we have $(V+\kappa)_+\in L^1_\loc(\Omega)$ and, since $(V+\kappa)_-^{1+d/2}\lesssim \kappa^{-1+\gamma} V_-^{\gamma+d/2}$, $(V+\kappa)_-\in L^{1+d/2}(\Omega)$. Thus, by Theorem \ref{main} with $\gamma=1$, we have, for any $\kappa>0$,
$$
\lim_{h\to 0_+} h^d  \Tr \left( -h^2\Delta_\Omega + V + \kappa \right)_- = L_{\gamma,d}^\cl \int_\Omega (V+\kappa)_-^{1+d/2}\,dx \,.
$$
Moreover, by the Lieb--Thirring inequality (see, e.g., \cite[Theorem 4.38]{FrLaWe}), we have
$$
h^d  \Tr \left( -h^2\Delta_\Omega + V + \kappa \right)_- \lesssim \int_\Omega (V+\kappa)_-^{1+d/2}\,dx \,.
$$
Thus, by dominated convergence,
$$
\lim_{h\to 0_+} h^d \Tr \left( -h^2\Delta_\Omega + V \right)_-^\gamma
= \gamma(\gamma-1) \int_0^\infty L_{1,d}^\cl \int_\Omega (V+\kappa)_-^{1+d/2}\,dx\, \kappa^{\gamma-2}\,d\kappa \,.
$$
Since
\begin{equation}
	\label{eq:scconsts}
	\gamma(\gamma-1) \int_0^\infty L_{1,d}^\cl (V+\kappa)_-^{1+d/2} \kappa^{\gamma-2}\,d\kappa = L_{\gamma,d}^\cl V_-^{\gamma+d/2} \,,
\end{equation}
this implies the claimed asymptotics for $\gamma>1$.


\section{The case $\gamma<1$}\label{sec:<1}

Our goal in this section is to prove Theorem \ref{main} in the remaining case $0\leq\gamma<1$.


\subsection{A weak convergence result}\label{sec:weakconv}

The following assertion is a tool in the proof of Theorem \ref{main}, which is of independent interest. We recall that we introduced the density $\rho_D$ of a trace class operator $D$ in Subsection \ref{sec:lower} above.

\begin{theorem}\label{weylschr}
	Let $\gamma\geq 1$, let $\Omega\subset\R^d$ be an open set and let $V\in L^1_\loc(\Omega)$ with	$V_-\in L^{1+d/2}(\Omega)$. Then $D_h :=(-h^2\Delta+V)_-^{\gamma-1}$ satisfies, as $h\to 0$,
	\begin{equation}
		\label{eq:weylschrptw}
		h^d \rho_{D_h} \rightharpoonup L_{\gamma-1,d}^\cl \,V_-^{\gamma+d/2-1}
		\qquad\text{in}\ L^{(\gamma+d/2)'}(\Omega) \,.
	\end{equation}
\end{theorem}

Here, $(\gamma+d/2)' = (\gamma+d/2)/(\gamma+d/2-1)$ denotes the H\"older dual of $\gamma+d/2$ and, for $\gamma=1$, we use the notation $(-h^2\Delta+V)_-^0 = \1_{\{-h^2\Delta+V<0\}}$.

\begin{proof}[Proof of Theorem \ref{weylschr}]
	Fix $U\in L^{\gamma+d/2}(\Omega)$. We have to show that
	$$
	h^d \tr U D_h \to L_{\gamma-1,d}^\cl \int U V^{\gamma+d/2-1}_- \,dx \,.
	$$
	We have, for every $\lambda\in\R$,
	\begin{align*}
		\lambda \tr U D_h & = - \tr(-h^2\Delta_\Omega +V)D_h + \tr(-h^2\Delta_\Omega +V+\lambda U)D_h \\
		& = \tr(-h^2\Delta_\Omega +V)_-^\gamma + \tr(-h^2\Delta_\Omega +V+\lambda U)D_h \,.
	\end{align*}
	Let us bound the last term on the right side using the H\"older inequality for traces,
	\begin{align*}
		\tr(-h^2\Delta_\Omega +V+\lambda U)D_h & \geq - \tr(-h^2\Delta_\Omega +V+\lambda U)_-D_h \\
		& \geq - \left( \tr(-h^2\Delta_\Omega +V+\lambda U)_-^\gamma \right)^{1/\gamma} \left( \tr D_h^{\gamma'} \right)^{1/\gamma'} \\
		& =  - \left( \tr(-h^2\Delta_\Omega +V+\lambda U)_-^\gamma \right)^{1/\gamma} \left( \tr(-h^2\Delta_\Omega +V)_-^\gamma \right)^{1/\gamma'}.
	\end{align*}
	We insert this into the above equation with $\lambda=\pm\epsilon$, where $\epsilon>0$, and obtain
	\begin{align*}
		& \frac{1}{\epsilon}
		\left( \tr(-h^2\Delta_\Omega +V)_-^\gamma - \left( \tr(-h^2\Delta_\Omega +V+\epsilon U)_-^\gamma \right)^{1/\gamma} \left( \tr(-h^2\Delta_\Omega +V)_-^\gamma \right)^{1/\gamma'} \right) \\
		& \quad \leq \tr U D_h \\
		& \quad \leq \frac{1}{\epsilon}
		\left( \left( \tr(-h^2\Delta_\Omega +V-\epsilon U)_-^\gamma \right)^{1/\gamma} \left( \tr(-h^2\Delta_\Omega +V)_-^\gamma \right)^{1/\gamma'} -\tr(-h^2\Delta_\Omega +V)_-^\gamma \right).
	\end{align*}
	Multiplying by $h^d$ and letting $h\to 0_+$ gives, in view of Theorem \ref{eq:main} (which we have already proved for $\gamma\geq 1$),
	\begin{align*}
		& \frac{1}{\epsilon} L_{\gamma,d}^\cl \left( \int_\Omega V_-^{\gamma+d/2}\,dx - \left( \int_\Omega (V+\epsilon U)_-^{\gamma+d/2} \,dx \right)^{1/\gamma} \left( \int_\Omega V_-^{\gamma+d/2}\,dx \right)^{1/\gamma'} \right) \\
		& \quad \leq \liminf_{h\to 0_+} h^d \tr U D_h \leq \limsup_{h\to 0_+} h^d \tr U D_h \\
		& \quad \leq
		\frac{1}{\epsilon} L_{\gamma,d}^\cl \left( \left( \int_\Omega (V-\epsilon U)_-^{\gamma+d/2} \,dx \right)^{1/\gamma} \left( \int_\Omega V_-^{\gamma+d/2}\,dx \right)^{1/\gamma'}
		- \int_\Omega V_-^{\gamma+d/2} \,dx \right).
	\end{align*}
	The assertion now follows as $\epsilon\to 0_+$ by dominated convergence, using the fact that $L_{\gamma,d}^\cl (\gamma+d/2)/\gamma = L_{\gamma-1,d}^\cl$.
\end{proof}

Let us now turn to the proof of Theorem \ref{main} in the case $\gamma<1$, which we split again into an upper and a lower bound.


\subsection{Lower bound on $\tr(-h^2\Delta_\Omega +V)_-^\gamma$}

First, let $V_-\in L^{\gamma+d/2}\cap L^{\gamma+d/2+1}(\Omega)$. Let $E\subset\Omega$ be a set of finite measure and bound, with $\tilde D_h:=(-h^2\Delta+V)_-^\gamma$,
$$
\Tr(-h^2\Delta_\Omega +V)_-^\gamma = \int_\Omega \rho_{D_h} \,dx \geq \int_E \rho_{D_h} \,dx \,.
$$
Since $\1_E\in L^{(\gamma +d/2+1)'}(\Omega)$, Theorem \ref{weylschr} implies that
$$
\liminf_{h\to 0_+} h^d \Tr(-h^2\Delta_\Omega +V)_-^\gamma \geq \liminf_{h\to 0_+} h^d \int_E \rho_{D_h} \,dx = L_{\gamma,d}^\cl \int_E V_-^{\gamma+d/2}\,dx \,.
$$
If $\Omega$ has finite measure, we can take $E=\Omega$ and we are done. If $\Omega$ has infinite measure, we can take a sequence of $E$'s that increase to $\Omega$ to obtain the claimed bound.

Let us remove the additional assumption $V_- \in L^{\gamma+d/2+1}(\Omega)$. For $M>0$ and $V_M$ as in \eqref{eq:defvm}, we have
$$
\Tr(-h^2\Delta_\Omega +V)_-^\gamma \geq \Tr(-h^2\Delta_\Omega +V_M)_-^\gamma \,.
$$
Since $(V_M)_-\in L^{\gamma+d/2+1}(\Omega)$, we can apply the previous result and obtain
$$
\liminf_{h\to 0_+} h^d \Tr(-h^2\Delta_\Omega +V)_-^\gamma \geq L_{\gamma,d}^\cl \int_{\R^d} (V_M)_-^{\gamma+d/2}\,dx \,.
$$
This implies the claimed lower bound by letting $M\to\infty$.

Note that in this argument we did not use the Lieb--Thirring inequality, thereby proving Remark (b) after Theorem \ref{main}.


\subsection{Upper bound on $\tr(-h^2\Delta_\Omega +V)_-^\gamma$}\label{sec:uppergamma}

The proof of the upper bound is also based on Theorem \ref{weylschr}, but, similarly as in the proof of Theorem \ref{main} for $\gamma=1$, an additional localization argument is needed.

Similarly as in the proof of the lower bound, we first assume $V_-\in L^{\gamma+d/2}\cap L^{\gamma+d/2+1}(\Omega)$. Given $R_->0$, we choose two smooth, real-valued functions $\chi_<$ and $\chi_>$ on $\R^d$ such that $\chi_<^2+\chi_>^2\equiv 1$, such that $\chi_<$ has compact support and such that $\chi_<\equiv 1$ on $B_{R_-}$. By the IMS formula, we have
\begin{align*}
	-h^2\Delta+V & = \chi_< \left(-h^2\Delta_\Omega +V_h \right)\chi_< + \chi_> \left(-h^2\Delta_\Omega +V_h \right)\chi_> \,.
\end{align*}
with $V_h$ from \eqref{eq:imspot}. This, together with a simple consequence of the variational principle (see, e.g., \cite[Proposition 1.40]{FrLaWe}), gives that for any $0<\theta<1$,
\begin{align}\label{eq:rotfeld}
	\tr(-h^2\Delta+V)_-^\gamma & \leq \theta^{-\gamma} \tr\left( \chi_< \left(-h^2\Delta_\Omega +V_h \right)\chi_< \right)_-^\gamma \notag \\
	& \quad + (1-\theta)^{-\gamma} \tr\left( \chi_> \left(-h^2\Delta_\Omega +V_h \right)\chi_> \right)_-^\gamma.
\end{align}
(Indeed, it follows from Rot'feld's inequality \cite[Proposition 1.43]{FrLaWe}) that both constants $\theta^{-\gamma}$ and $(1-\theta)^{-\gamma}$ can be replaced by $1$, but we do not need this much more subtle result.)

Let us discuss the two terms on the right side of \eqref{eq:rotfeld} separately. We begin with the first one. Let $R_+<\infty$ be such that the support of $\chi_<$ is contained in the ball $B_{R_+}$ and put $\omega:=\Omega\cap B_{R_+}$. Then, as a consequence of the variational principle (see, e.g., \cite[Corollary 1.31 and Lemma 1.44]{FrLaWe}),
\begin{align*}
	\tr\left( \chi_< \left(-h^2\Delta_\Omega +V_h \right)\chi_< \right)_-^\gamma
	& \leq \tr \left( \chi_< \left(-h^2\Delta_\omega + \1_\omega V_h \right)\chi_< \right)_-^\gamma \\
	& \leq \tr \left(-h^2\Delta_\omega + \1_\omega V_h \right)_-^\gamma \,.
\end{align*}
For given $h_0>0$ we set $D_h':=\left( -h^2\Delta_\omega + \1_\omega V_{h_0} \right)_-^\gamma$. Then for all $h\leq h_0$
\begin{align*}
	\tr \left(-h^2\Delta_\omega + \1_\omega V_h \right)_-^\gamma
	\leq \tr \left(-h^2\Delta_\omega + \1_\omega V_{h_0} \right)_-^\gamma = \int_\omega D_h'\,dx \,.
\end{align*}
By Theorem \ref{weylschr}, we conclude that
\begin{align*}
	\limsup_{h\to 0_+} h^d \tr \left(-h^2\Delta_\omega + \1_\omega V_h \right)_-^\gamma
	& \leq \limsup_{h\to 0_+} h^d \int_\omega \rho_{D_h'}\,dx \\
	& = L_{\gamma,d}^\cl \int_\omega \left( V_{h_0}\right)_-^{\gamma+d/2} \,dx \,.
\end{align*}
Since $h_0>0$ is arbitrary, monotone convergence allows us to replace $V_{h_0}$ on the right side by $V$. Thus, we have shown that
$$
\limsup_{h\to 0_+} h^d \tr\left( \chi_< \left(-h^2\Delta_\Omega +V_h \right)\chi_< \right)_-^\gamma
\leq L_{\gamma,d}^\cl \int_\Omega V_-^{\gamma+d/2} \,dx \leq L_{\gamma,d}^\cl \int_{\R^d} V_-^{\gamma+d/2} \,dx \,.
$$

We turn now to the second term on the right side of \eqref{eq:rotfeld}. Similarly as before, we have
$$
\tr\left( \chi_> \left(-h^2\Delta_\Omega +V_h \right)\chi_> \right)_-^\gamma
\leq \tr\left( -h^2\Delta_\Omega + \1_{\Omega\cap B_{R_-}^c} V_h \right)_-^\gamma
$$
According to the Lieb--Thirring inequality (which is valid under our assumptions on $\gamma$), the right side is bounded by
$$
h^{-d} L_{\gamma,d} \int_{\Omega\cap B_{R_-}^c} \left( V_h\right)_-^{\gamma+d/2}\,dx 
= h^{-d} L_{\gamma,d} \left( \int_{\Omega\cap B_{R_-}^c} V_-^{\gamma+d/2}\,dx + o(1) \right). 
$$

Returning to \eqref{eq:rotfeld}, we have shown that
$$
\limsup_{h\to 0_+} h^d \tr(-h^2\Delta_\Omega +V)_-^\gamma \leq L_{\gamma,d}^\cl \int_\Omega V_-^{\gamma+d/2}\,dx + L_{\gamma,d} \int_{\Omega\cap B_{R_-}^c} V_-^{\gamma+d/2}\,dx \,.
$$
Leting $R_-\to\infty$ and using dominated convergence, we deduce the claimed asymptotic upper bound under the additional assumption $V_- \in L^{\gamma+d/2+1}(\Omega)$.

Finally, this additional assumption can be removed using the Lieb--Thirring inequality as, for instance, in \cite[Theorem 4.46]{FrLaWe}. This completes the proof of Theorem \ref{main} in the remaining case $0\leq\gamma<1$.


\section{Extension to the magnetic case}\label{sec:magnetic}

It is of interest to study the corresponding questions in the presence of a magnetic field. Let again $\Omega\subset\R^d$ be an open set and let $A\in L^2_\loc(\Omega,\R^d)$. We consider operators of the form
$$
M_h + V = (-ih\nabla +g_hA)^2 + V
\qquad\text{in}\ L^2(\Omega) \,,
$$
where we assume through this section that the coupling constant $g_h\in\R$ satisfies
$$
\limsup_{h\to 0_+} |g_h| <\infty \,.
$$
There are at least two natural choices of $g_h$, namely, $g_h=1$ and $g_h=h$, which arise in applications. It presents no extra effort, however, to deal with the general case.

Concerning $V$ we make the same assumptions as in Theorem \ref{main}, that is, $V\in L^1_\loc(\Omega)$ and $V_-\in L^{\gamma+d/2}(\Omega)$ for $\gamma$ as specified there. Technically, the operator $M_h+V$ is defined through the closure of the quadratic form
$$
\int_\Omega \left( |(-ih\nabla + g_h A)\psi|^2 + V|\psi|^2\right)dx
$$
defined on $C_c^\infty(\Omega)$. Functions $\psi$ in the domain of this closure satisfy $(-ih\nabla+g_hA)\psi\in L^2(\Omega)$ and $V_+|\psi|^2\in L^2(\Omega)$, as well as, in a certain sense, Dirichlet boundary conditions on $\partial\Omega$.

The following result is the extension of Theorem \ref{main} to the magnetic case.

\begin{theorem}\label{mainmag}
	Let $\gamma\geq 1/2$ if $d=1$, $\gamma>0$ if $d=2$ and $\gamma\geq 0$ if $d\geq 3$. Let $\Omega\subset\R^d$ be an open set, let $V\in L^1_\loc(\Omega)$ with $V_-\in L^{\gamma+d/2}(\Omega)$ and let $A\in L^2_\loc(\Omega,\R^d)$. Then
	\begin{equation}
		\label{eq:mainmag}
		\lim_{h\to 0_+} h^d \Tr \left( M_h + V \right)_-^\gamma = L_{\gamma,d}^\cl \int_\Omega V(x)_-^{\gamma+d/2}\,dx \,.
	\end{equation}
\end{theorem}

\emph{Remarks.}
(a) As an example, note that in 2D the case of an Aharonov--Bohm vector potential is included. Indeed, the assumptions are satisfied with $A(x) = |x|^{-2}(-x_2,x_1)^{\rm T}$ and $\Omega=\R^2\setminus\{0\}$.\\
(b) In \cite{Fr0} we have shown this theorem in case $V$ is a negative constant, $g_h=h$ and $\Omega$ has finite measure. The proof uses a Tauberian theorem and relies on the fact that one essentially deals with the eigenvalues of a single, $h$-independent operator. We do not see how to extend this proof to a more general setting.\\
(c) The proof of \eqref{eq:mainmag} splits again into two parts, namely,
\begin{equation}
	\label{eq:mainmaglower}
	\liminf_{h\to 0_+} h^d \Tr \left( M_h + V \right)_-^\gamma \geq L_{\gamma,d}^\cl \int_\Omega V(x)_-^{\gamma+d/2}\,dx
\end{equation}
and
\begin{equation}
	\label{eq:mainmagupper}
	\limsup_{h\to 0_+} h^d \Tr \left( M_h + V \right)_-^\gamma \leq L_{\gamma,d}^\cl \int_\Omega V(x)_-^{\gamma+d/2}\,dx \,.
\end{equation}
That of \eqref{eq:mainmaglower} is analogous to the corresponding proof for Theorem \ref{main}. For the proof of \eqref{eq:mainmagupper}, however, we have not been able to find an analogous way. Instead, we rely on the sharp Lieb--Thirring inequality of Laptev and Weidl \cite{LaWe}, which is also valid in the magnetic case. The fact that this inequality is only valid for $\gamma\geq 3/2$ necessitates some extra arguments. In particular, it is crucial here that the weak convergence argument is not limited to $\gamma=1$. Thus, we give also a partially alternate proof of \eqref{eq:mainupper}, although we feel that in the nonmagnetic case our first proof is more natural.\\
(d) Our proof shows that \eqref{eq:mainmaglower} holds for all $\gamma\geq 0$, provided the operators $M_h+V$ are well-defined.\\
(e) Corollary \ref{maincor} remains valid in the magnetic case, with the same proof.


\subsection{The upper bound on $\Tr \left( M_h + V \right)_-^\gamma$ for $\gamma\geq 3/2$}

In this subsection, prove \eqref{eq:mainmagupper} for $\gamma\geq 3/2$. In fact, in this case we have the nonasymptotic result
\begin{equation}
	\label{eq:ltlw}
	h^d\, \tr\left( M_h +V\right)_-^\gamma \leq L_{\gamma,d}^\cl \int_\Omega V(x)_-^{\gamma+d/2}\,dx \,.
\end{equation}
To prove this, by the variational principle, we may and will assume that $V\leq 0$. In the case $\Omega=\R^d$, this is \cite[Theorem 3.2]{LaWe}. By the variational principle, this implies the result for arbitrary open sets $\Omega\subset\R^d$ and $A\in L^2_{\loc}(\overline\Omega,\R^d)$. To deal with the general case where $A\in L^2_\loc(\Omega,\R^d)$ we need an approximation argument. We use the fact that, if $A_n\to A$ in $L^2_\loc(\Omega,\R^d)$, then the corresponding operators $M_{h,n}$ converge in strong resolvent sense to $M_h$. This can be proved in the same way as in \cite[Lemma 5]{LeSi}. The strong convergence implies that $\left( M_{h,n} +V\right)_-^\gamma$ converges strongly to $\left( M_h +V\right)_-^\gamma$; see \cite[Theorem VIII.20]{ReSi}. Here we also use the fact that the operators $M_{h,n}+V$ are bounded from below uniformly in $n$, so the function $f(\lambda)=\lambda_-^\gamma$ can be truncated near $-\infty$. Now the lower semicontinuity of the trace under weak convergence implies
$$
\tr \left( M_h +V\right)_-^\gamma \leq \liminf_{n\to\infty} \left( M_{h,n} +V\right)_-^\gamma \,.
$$
Therefore, if we assume $A_n\in L^2_\loc(\overline\Omega,\R^d)$, then the inequality for $M_{h,n}$ implies the claimed inequality for $M_h$.


\subsection{The lower bound on $\Tr \left( M_h + V \right)_-^\gamma$ for $\gamma\geq 1$}

In this subsection, we prove \eqref{eq:mainmaglower} for $\gamma\geq 1$. 

\subsubsection*{The case $\gamma=1$}

We follow very closely the arguments in Subsection \ref{sec:upper}. We fix again $\omega$, $g$ and $M$. Now we choose as a trial matrix
$$
\Gamma := \iint_{\R^d\times\omega} \1_{\{(hp+g_h \langle A \rangle(q))^2 + V_M(q)<0\}} |G_{p,q}\rangle\langle G_{p,q}| \,\frac{dp\,dq}{(2\pi)^d} \,.
$$
The magnetic analogue of \eqref{eq:cslowerkin} is
\begin{equation}
	\label{eq:cslowerkinmag}
	\|(-i \nabla + A) G_{p,q}\|^2 = p^2 + 2 p \cdot \langle A \rangle (q) + \langle A^2 \rangle (q) + \|\nabla g\|^2 \,,
\end{equation}
which is proved by a similar argument. Using this, we obtain, as before,
\begin{align*}
	& -\tr\left( M_h +V\right)_- \leq \tr \left( M_h + V \right)\Gamma \\
	& \quad = \iint_{\R^d\times\omega} \1_{\{ (hp+g_h\langle A \rangle(q))^2 + V_M(q)<0\}} \left( (h p+g_h \langle A \rangle(q))^2 + V_M(q) \right) \frac{dp\,dq}{(2\pi)^d} + h^{-d} \mathcal R^{\rm mag} \\
	& \quad = - h^{-d} L_{1,d}^\cl \int_{\omega} V_M(q)_-^{1+d/2} \,dq +h^{-d} \mathcal R^{\rm mag} \,.
\end{align*}
Here we have set $\mathcal R^{\rm mag} = \mathcal R + \mathcal R_3$ with $\mathcal R$ as in the nonmagnetic case and with
\begin{align*}
	\mathcal R_3 & := - h^d g_h^2 \iint_{\R^d\times\omega} \1_{\{ (hp+g_h\langle A \rangle(q))^2 + V_M(q)<0\}} \left( \langle A\rangle(q)^2 - \langle A^2 \rangle(q) \right) \frac{dp\,dq}{(2\pi)^d} \\
	& = - g_h^2 L_{0,d}^\cl \int_\omega V_M(q)_-^{d/2} \left( \langle A\rangle(q)^2 - \langle A^2 \rangle(q) \right) dq \,.
\end{align*}
The integral here is finite, since, by the support properties of $g$, it only depends on the restriction of $A$ to a neighborhood of $\omega$ that is compactly contained in $\Omega$. Thus, when $\lim_{h\to 0_+} g_h=0$, the inequality \eqref{eq:mainmaglower} follows immediately from the argument in the nonmagnetic case.

In the general case, we argue similarly as for the term $\mathcal R_2'$ in Subsection \ref{sec:upper} by letting the support of $g$ tend to zero. More precisely, we find that $\langle A \rangle_\delta \to A$ in $L^2(\omega)$ and $\langle A^2 \rangle_\delta \to A^2$ in $L^1(\omega)$. Since $(V_M)_-\in L^\infty(\omega)$, this proves that $\lim_{\delta\to 0} \mathcal R_2'=0$. This completes the proof of \eqref{eq:mainmaglower} for $\gamma=1$.


\subsubsection*{The case $\gamma>1$}

We argue similarly as in Section \ref{sec:>1} based on the magnetic analogue of \eqref{eq:representation}. By Fatou's lemma, this formula together with the $\gamma=1$ case of \eqref{eq:mainmaglower} implies, for $\gamma>1$,
\begin{align*}
	\liminf_{h\to 0_+} h^d \tr\left( M_h +V \right)_-^\gamma & \geq \gamma(\gamma-1) \int_0^\infty \liminf_{h\to 0_+} h^d \tr\left( M_h +V+\kappa \right)_- \kappa^{\gamma-2}\,d\kappa \\
	& \geq \gamma(\gamma-1) \int_0^\infty L_{\gamma,d}^\cl \int_\Omega (V+\kappa)_-^{\gamma+d/2}\,dx\,\kappa^{\gamma-2}\,d\kappa \,.
\end{align*}
By \eqref{eq:scconsts}, this gives \eqref{eq:mainmaglower} for $\gamma>1$.


\subsection{Proof of Theorem \ref{mainmag}}

It follows from the previous two subsections that \eqref{eq:mainmag} holds for all $\gamma\geq 3/2$.

By the same argument as in the proof of Theorem \ref{weylschr}, for $\gamma\geq 3/2$, we obtain the weak convergence
$$
h^d \rho_{D_h} \rightharpoonup L_{\gamma-1,d}^\cl V_-^{\gamma+d/2-1}
\qquad\text{in}\ L^{(\gamma+d/2)'}(\Omega)
$$
for $D_h := (M_h+V)_-^{\gamma-1}$. Based on this weak convergence, we can continue to argue as in Section \ref{sec:<1} and deduce that \eqref{eq:mainmag} holds for all $\gamma\geq 1/2$. Note that here in the analogue of the approximation argument at the end of Subsection \ref{sec:uppergamma} we use the fact that the Lieb--Thirring inequality holds with a constant independent of the magnetic field; see, e.g., \cite[Section 4.9]{FrLaWe}.

Now we can iterate this argument. From the validity of \eqref{eq:mainmag} for all $\gamma\geq 1$ we infer that the weak convergence holds in fact for $\gamma\geq 1$ and then we can conclude as before that \eqref{eq:mainmag} holds in the full range claimed in Theorem \ref{mainmag}. This completes its proof.\qed


\bibliographystyle{amsalpha}

\end{document}